\documentclass[10pt,leqno]{amsart}
\usepackage{amsmath}
\usepackage{amsfonts}
\usepackage{amsthm}
\usepackage{amsbsy}
\usepackage{amssymb}
\usepackage{epsfig}
\usepackage{latexsym}
\usepackage{setspace}

\newcommand\beque{\begin{equation*}}
\newcommand\beq{\begin{equation}}
\newcommand\eeq{\end{equation}}
\newcommand\eeque{\end{equation*}}
\newcommand\eqn{\begin{eqnarray}}
\newcommand\beqna{\begin{eqnarray*}}
\newcommand\eeqna{\end{eqnarray*}}
\newcommand\feqn{\end{eqnarray}}

\newcommand{\nn}{\nonumber}

\newcommand\RR{{\mathbb R}}

\newtheorem{lemma}{Lemma}
\newtheorem{propo}{Proposition}
\newtheorem{theorem}{Theorem}

\newtheorem{theorem2}{Theorem}

\newtheorem{corollary}{Corollary}

\numberwithin{equation}{section}

\begin{document}



\date{\today}

\title[]{\textbf{\Large{Uniformization, Unipotent Flows and the Riemann Hypothesis}}}

\author{Sergio L. Cacciatori and  Matteo A. Cardella}

\address{Dipartimento di Fisica e Matematica, Universit\`a dell'Insubria, via Valleggio 11, 22100
Como, Italy, and\\
INFN, sezione di Milano, via Celoria 16, 20133 Milano, Italy}
\email{sergio.cacciatori@mi.infn.it}

\address{Theoretical Physics Department,  Universit\`a di Milano Bicocca,
Piazza della Scienza 3, I-20126 Milano, Italy,}
\email{matteo@phys.huji.ac.il}

\begin{abstract}
We prove equidistribution of certain multidimensional unipotent flows
in the moduli space of  genus $g$ principally polarized abelian varieties (ppav).
This is done by  studying  asymptotics of $\pmb{\Gamma}_{g} \sim  Sp(2g,\mathbb{Z})$-automorphic forms averaged along  unipotent flows,
toward the codimension-one  component of the boundary of the ppav moduli space.
We prove a link between the  error estimate  and  the Riemann hypothesis.
Further, we prove  $\pmb{\Gamma}_{g - r}$ modularity of the   function obtained by iterating
the unipotent average process $r$ times. This  shows  uniformization
of  modular integrals of automorphic functions via unipotent flows.

\end{abstract}

\maketitle

\section*{\Large{\textbf{Introduction}}}

\vspace{.27  cm}

Let $\pmb{\mathcal{H}}_g : = \{ \tau \in Mat(g, \mathbb{C}) \, |  \, \tau  = \tau^t , \, \, \Im(\tau) > 0 \}$,
the genus $g$ Siegel half space, i.e. the set of  symmetric complex $g \times g$ matrices $\tau$, with positive definite imaginary part $\Im(\tau)$.
 We indicate with $\pmb{\Gamma}_g \sim Sp(2g,\mathbb{Z})$ the discrete  group of symplectic
transformations, with action on $\tau$ given by

\vspace{.27 cm}

\beq
\tau \rightarrow  (a \tau + b) (c \tau + d)^{-1}, \qquad \begin{pmatrix} a & b \\ c & d   \end{pmatrix} \in Sp(2g,\mathbb{Z}). \nn
\eeq

\vspace{.27 cm}

The coset space $\pmb{\Gamma}_g \backslash \pmb{\mathcal{H}}_g $ is the moduli space
 of   genus $g$ principally polarized abelian varieties (ppav) $\pmb{\mathcal{A}}_g$, and for genera $g = 1,2,3$, $\pmb{\mathcal{A}}_g$
 is isomorphic  to  the moduli space of compact Riemann surfaces $\pmb{\mathcal{M}}_g$.

$\pmb{\mathcal{H}}_g$ is a homogenous space, since it is isomorphic  to the Lie coset $ \pmb{\mathcal{H}}_g \simeq Sp(2g,\mathbb{R}) / \left( Sp(2g,\mathbb{R}) \cap SO(2g, \mathbb{R}) \right)$, the set of   real symplectic matrices over the orthosymplectic ones.

By Iwasawa decomposition of a symplectic matrix in  $Sp(2g,\mathbb{R})$, one can find an interesting set of coordinates
for the Lie coset  $\pmb{\mathcal{A}}_g$.
 Every  symplectic matrix can be written as  $\pmb{U} \pmb{A} \pmb{K}$, with $ \pmb{K} \in   Sp(2g,\mathbb{R}) \cap SO(2g, \mathbb{R})$,
$\pmb{A} = {\rm{diag}}( \pmb{V}_{g} , \pmb{V}_{g}^{-1})$, diagonal $2g \times 2g$ matrix, $(\pmb{V}_{g} : = {\rm{diag}}(\sqrt{v_1},\dots,\sqrt{v}_g ), v_i > 0, i = 1,\dots,g)$, and

\vspace{.27 cm}

$$\pmb{U} = \begin{pmatrix} \pmb{U}_{g} & \, \, \pmb{W}_{g}\pmb{U}_{g}^{-t} \\ 0   & \, \, \pmb{U}_{g}^{-t} \end{pmatrix}$$

\vspace{.27 cm}

is a $2g \times 2g$ real unipotent  matrix,
with $\pmb{U}_{g}$ upper unitriangular $g \times g$ real matrix, and $\pmb{W}_{g}$   symmetric $g \times g$ real matrix.

\vspace{.27 cm}

This leads to the following Iwasawa parametrization of the Siegel Half space

\vspace{.27 cm}

\beq
\tau_g = \pmb{W}_{g} + i \pmb{U}_{g} \pmb{V}^{2}_g \pmb{U}_{g}^{-t}. \label{Iwtau}
\eeq

\vspace{.27 cm}

 In this paper we  exploit    Iwasawa parametrization, in conjunction with  the Rankin-Selberg method,
 for investigating properties of a certain class of unipotent  averages  of automophic forms
 on $\pmb{\Gamma}_g \backslash \pmb{\mathcal{H}}_g $.
From the behavior of those averages in the   asymptotics  limit  toward
the  codimension one  component of the  $ \pmb{\Gamma}_g \backslash \pmb{\mathcal{H}}_g $ boundary,
 we prove    ergodicity of  (multidimensional) unipotent flows.
It turns out that the error estimate   depends  on $\Theta : = {\rm{Sup}}\{\Re(\rho)| \zeta^{*}(\rho) = 0 \}$,
the superior of the real part of the non trivial zeros of the Riemann
zeta function. Therefore, evaluation of  the error estimate would prove or disprove  the Riemann hypothesis. Our  result
  generalizes a well known theorem by  Zagier  on  the long  horocycle average of a $SL(2,\mathbb{Z})$ automorphic  functions of rapid decay [Za1],[Za2].

\vspace{.27 cm}

In order to  announce
the two main results of this paper, we need to introduce  further notations.
Given $\tau_g \in \pmb{\mathcal{H}}_g$, we use a $(g - r)$-corank block  decomposition
$\tau_g = \begin{pmatrix} \tau_r & \tau_2 \\ \tau_{2}^t &  \tau_{g - r}  \end{pmatrix}$,
with $\tau_r \in \pmb{\mathcal{H}}_{r}$, $\tau_{g - r} \in \pmb{\mathcal{H}}_{g - r}$.

\vspace{.27 cm}

In the $r = 1$ case,  the Iwasawa coordinatization (\ref{Iwtau}) gives

\vspace{.27 cm}

\beq
\tau_g = \begin{pmatrix} \tau_1 & \tau_2 \\ \tau_{2}^t &  \tau_{g - 1}  \end{pmatrix}
= \begin{pmatrix}
w_{11} + i(v_{1}  + \pmb{\underline{u}} \pmb{V}_{g - 1}^{2} \pmb{\bar{u}})   & \pmb{\underline{w}} +  i \pmb{\underline{u}}\pmb{V}_{g - 1}^{2}\pmb{U}_{g - 1}^{t} \\
\pmb{\bar{w}}  +   i \pmb{U}_{g - 1}\pmb{V}_{g - 1}^{2} \pmb{\bar{u}}   &   \tau_{g -1}
\end{pmatrix}, \label{Iwcorank1}
\eeq

\vspace{.27 cm}

where $\pmb{\underline{w}} : (w_{12} , \dots  w_{1g})$ is  in the first row of the symmetric real matrix $\pmb{W}_{g}$,
$\pmb{\underline{u}} : = (u_{12}, \dots u_{1g})$ is  in  the first row
of the unitriangular real matrix $\pmb{U}_{g}$. We also use the following notation $\pmb{\bar{w}} : = \pmb{\underline{w}}^t$, $\pmb{\bar{u}} : = \pmb{\underline{u}}^t$
to denote column vectors.

\vspace{.27 cm}
The first main result of this paper   concerns modularity under  the subgroup of transformations   $\pmb{\Gamma}_{g-1}$
of the average of an automorphic function $f(\tau)$  along the $(2g - 1)$  unipotent  directions
$(w_{11}, \pmb{\underline{w}}, \pmb{\underline{u}})$:

 \vspace{.27 cm}

\begin{theorem}  \label{invariance}  
Given a $\pmb{\Gamma}_g$-invariant automorphic function $f = f(\tau)$,  let us consider the unipotent average
\beq
\pmb{<} f \pmb{>}_{v_1}(\tau_{g - 1}) : = \int_{\mathbb{R}^{2g - 1}} d w_{11} d \pmb{\underline{w}} \, d \pmb{\underline{u}} \, f(\tau_g ), \nn
\eeq
where $\tau_g$ is given in Iwasawa coordinates according to the corank $(g - 1)$ decomposition given in (\ref{Iwcorank1}).

\vspace{ .27 cm}

The integral  function $\pmb{<} f \pmb{>}_{v_1}(\tau_{g - 1})$ on $\mathbb{R}_{> 0} \times \pmb{\mathcal{H}}_{g - 1}$
is invariant under the genus $(g - 1)$ modular group  $\pmb{\Gamma}_{g - 1}$:

\beq
\pmb{<} f \pmb{>}_{v_1}( (a\tau_{g - 1} + b)(c\tau_{g-1} + d)^{-1}) = \pmb{<} f \pmb{>}_{v_1}(\tau_{g - 1}), \qquad  \begin{pmatrix} a & b \\ c & d \end{pmatrix} \in \pmb{\Gamma}_{g - 1}. \nn
\eeq

\end{theorem}

\vspace{.4 cm}

The $\pmb{\mathcal{H}}_g$ boundary is given by $g - 1$  components $\mathbb{F}_{g - r}$, $r = 1,\dots, g - 1$.
For the quotient space  $\pmb{\Gamma}_g \backslash \pmb{\mathcal{H}}_g$ the  $(g - r)$-corank  boundary component is
given by
\beq
\pmb{\Gamma}_g \backslash \mathbb{F}_{g - r} =  \Big\{ \begin{pmatrix} i\infty_{r}  & 0 \\ 0 &  \tau_{g - r} \end{pmatrix}, \, \tau_{g - r} \in \pmb{\mathcal{H}}_r      \Big\}, \nn
\eeq

where $i\infty_{r} : = {\rm{diag}}(i\infty, \dots i\infty)$  represents  $r$ copies of the $\pmb{\Gamma}_{1}  \backslash\ \pmb{\mathcal{H}}_{1}$
cusp.  Theorem \ref{invariance} shows that  the unipotent average
$\pmb{<} f \pmb{>}_{v_1}( \tau_{g - 1})$  is a $\pmb{\Gamma}_{g - 1}$ invariant modular function, defined on
 $ \mathbb{R}_{> 0} \times \pmb{\mathcal{H}}_{g - 1}$. This function  can be thought to be   related to the $(g - 1)$-corank component $\mathbb{F}_{g-1}$ of
the  boundary of  $\pmb{\mathcal{H}}_g$.  The distance from this boundary component is controlled by
$v_1 > 0$, and one recovers the average along  the $\mathbb{F}_{g-1}$ component of the boundary in the $v_1 \rightarrow 0$ limit.

\vspace{.27 cm}

The second main result of this paper is given by theorem \ref{theorem} below.
Theorem \ref{theorem}  shows
that in the limit $v_1 \rightarrow 0$, the $\pmb{<} f \pmb{>}_{v_1}( \tau_{g - 1})$ averaged on the modular domain
$\pmb{\Gamma}_{g - 1} \backslash \pmb{\mathcal{H}}_{g - 1}$
converges to the $f$ average on the modular domain $\pmb{\Gamma}_g \backslash \pmb{\mathcal{H}}_g$.
The  error term is  related to the non trivial zeros of the Riemann zeta function $\pmb{\zeta}(s)$, and an estimate of this quantity
provides  a proof  (or a  disproof)  of   the Riemann hypothesis:

\vspace{.35 cm}

\newpage

\begin{theorem} \label{theorem} Let $f =  f(\tau)$ a $\pmb{\Gamma}_{g}$-invariant function of rapid decay for $\tau$ going to all the  components  $F_{g - r}, \, \,  r = 1,\dots,g-1$  of the
$\pmb{\mathcal{H}}_g$ boundary. Let $f(\tau)$ be differentiable   up to second order, with  Laplacian $\Delta f$  of rapid decay,
then the following asymptotic  holds true:

\vspace{.35 cm}

\beq
 \int_{ \pmb{\mathcal{D}}_{g - 1}}    d\pmb{\mu}_{g - 1} \pmb{<} f \pmb{>}(v_1 , \tau_{g-1})
 \sim \frac{Vol(\pmb{\mathcal{D}}_{g-1})}{2Vol(\pmb{\mathcal{D}}_{g})}\int_{ \pmb{\mathcal{D}}_g} d\pmb{\mu}_{g} \, f(\tau) +  O(v_{1}^{g - \frac{\Theta}{2}}),
 \quad  v_{1} \rightarrow 0. \nn
 \eeq

 \vspace{.35 cm}
Here  $\pmb{\mathcal{D}}_g \sim  \pmb{\Gamma}_g  \backslash \pmb{\mathcal{H}}_g $ is a $\pmb{\Gamma}_g$ fundamental domain,
with volume given by the formula $Vol(\pmb{\mathcal{D}}_g) = 2 \prod_{k=1}^{g}\pmb{\zeta}^{*}(2k)$, and
 $\Theta : = \sup\{\Re(\rho)| \pmb{\zeta}^{*}(\rho) =  0 \}$
is the superior of the real part of the non trivial zeros $\rho$'s of the Riemann zeta function $\pmb{\zeta}(s)$.

\end{theorem}

\vspace{.35 cm}

Theorem \ref{theorem} provides a  quite interesting  connection between asymptotic unipotent dynamics   in
the ppav moduli space $ \pmb{\Gamma}_g \backslash \pmb{\mathcal{H}}_g$
and the Riemann hypothesis. Indeed, in the genus $g$ case, the error estimate is $O(v_{1}^{g - 1/4})$
if and only if the Riemann hypothesis is true ($\Theta = 1/2$).  Theorem \ref{theorem} generalizes Zagier genus $g = 1$ result [Za1],[Za2],
for modular functions of rapid decay at the cusp.

\vspace{.35 cm}

Moreover, there is an interesting corollary of theorem  \ref{invariance}, which follows by iterating the operation of averaging
along unipotent directions. Let us use the following notation $\pmb{\underline{w}}^{(r)} : = (w_{r,r + 1}, \dots w_{r,g})$,
and $\pmb{\underline{u}}^{(r)} : = ( u_{r,r + 1}, \dots  u_{r,g}) $, where $w_{ij} : =  (\pmb{W_g})_{ij}$, $u_{ij} : = (\pmb{U_g})_{ij}$

\vspace{.35 cm}

\begin{corollary}
Let $f = f(\tau_{g})$ a $\pmb{\Gamma}_g$-invariant automorphic function.
For $r = 1,\dots, g - 1$, the following unipotent average

\beq
\pmb{<} f \pmb{>}_{v_{1},\dots, v_{r}} (\tau_{g - r}) : = \int_{\mathbb{R}_{> 0}^{g^{2} - r^{2}}} \, \prod_{i=1}^{r} dw_{ii}  \,  d\pmb{\underline{w}}^{(i)}  \, d\pmb{\underline{u}}^{(i)} \, f (\tau_{g}), \nn
\eeq

is a  $\pmb{\Gamma}_{g - r}$-invariant  function on $\mathbb{R}^{r}_{> 0} \times \pmb{\mathcal{H}}_{g - r}$:

\beq
\pmb{<} f \pmb{>}_{v_{1},\dots, v_{r}} ((a\tau_{g - r} + b)(c \tau_{g - r} + d)^{ -1}) =  \pmb{<} f \pmb{>}_{v_{1},\dots, v_{r}} (\tau_{g - r}), \qquad
\begin{pmatrix} a & b \\ c & d \end{pmatrix} \in \pmb{\Gamma}_{g - r}. \nn
\eeq

\end{corollary}

\vspace{.27 cm}

Results of this paper given in theorems \ref{invariance} and \ref{theorem} suggest interesting connections with
powerful  results on measure rigidity and equidistribution of unipotent flows, provided by Ratner theory  \cite{Ratner},
and by more recent developments,  (see for example   \cite{ELPV},\cite{EMV},\cite{EW}).

\vspace{.27 cm}

We also would like to mention an interesting connection between unipotent dynamics in homogeneous spaces
and string theory, \cite{CC},\cite{C1},\cite{C2},\cite{ACER}. In fact, results in this paper have  applications for shedding  light
in ultraviolet/infrared dualities descending from  finiteness of closed string perturbative amplitudes \cite{CC}.
For genus one, (one-loop), closed string amplitudes, equidistribution theorems for long horocycles in the modular surface
$ SL(2,\mathbb{Z}) \backslash SL(2,\mathbb{R})$ connects \cite{C1} vacuum stability with asymptotic supersymmetry \cite{KS}.
Moreover, graded spectra of closed string excitations exhibit  oscillating patterns with frequencies given by the imaginary
parts of the non trivial zeros of the Riemann zeta function \cite{ACER}. In one-loop stable closed string  vacua, asymptotic supersymmetry is maximal if and
only if the Riemann hypothesis is true \cite{ACER}. As suggested in \cite{CC}, equidistribution theorems for unipotent
averages in  the ppav moduli space when applied to higher  genus closed string amplitudes produce
generalizations of the one-loop result \cite{KS}.

\vspace{.27 cm}

The connection between homogenous space dynamics and string theory works also
in the opposite direction, namely from string theory to the theory of automorphic forms.
 By using consistency conditions from string theory, it
 provides information on certain asymptotic averages of automorphic forms.
The advantages of translating the dynamical problems in string theory terms has
been shown  in the specific case of the horocycle flow in \cite{C2}. There
 certain asymptotics for long horocycles averages of
modular invariant functions with not so mild growing conditions have been considered.
It is also worth to mention that results in \cite{CC} and those in this paper
  indicate an intriguing relation between
ultraviolet properties of closed strings on stable backgrounds and the Riemann
hypothesis. These relations thus extend beyond one loop order, where they were shown
to exist in \cite{ACER}.  Moreover, results of this paper may also be useful
for studying and probing non perturbative conjectures related to modularity
of the  effective string action  \cite{Bi}, \cite{Green1},\cite{Green2},\cite{West1},\cite{West2},\cite{Pioline1},\cite{Pioline2}.
They may find also  applications for genus $g = 2$ superstring amplitudes \cite{DHP},
and for  testing recent proposals for genus $g \ge 3$
closed string amplitudes  \cite{CDPvG},\cite{DPvG}, \cite{Piazza:2010ha}
\cite{Gr},\cite{Grushevsky:2008qp},\cite{Grushevsky:2008zp},\cite{Gaberdiel:2010jf},
\cite{MV},\cite{Matone:2005vm},\cite{Matone:2008td}, \cite{Mo}.

\vspace{.27 cm}

The organization of the rest of the paper is the following:
in  section \S \ref{SectionIw} we present some technical facts related to
Iwasawa coordinatization of $\pmb{\mathcal{H}}_g$, instrumental for our proofs.
Section \S \ref{SectionErg} contains the  proofs of theorems \ref{invariance} and \ref{theorem}  on ergodicity of unipotent flows and error estimates,
and some lemmas, instrumental for those  proofs. 


\vspace{.35 cm}

\section*{\bf{Acknowledgements}}

MC thanks  Gerard van der Geer  for enlighting  discussions.
The work of MC was supported at various stages by
the  Superstring Marie Curie Training
Network under the contract MRTN-CT-2004-512194
at the Hebrew University of Jerusalem, by a visiting fellowship at the ESI
Schroedinger Center for Mathematical Physics in Vienna,
by the Italian MIUR-PRIN contract 20075ATT78 at the University
of Milano Bicocca, by the Theory Unit at CERN, and by a
"Angelo Della Riccia" fellowship at the University of Amsterdam.

\vspace{.35 cm}


\section{\Large{\bf{Iwasawa parametrization for $\pmb{\mathcal{H}}_g$ and Eisenstein series}}}\label{SectionIw}

\vspace{.35 cm}

\subsection{Iwasawa parametrization for $\pmb{\mathcal{H}}_{g}$}
The genus $g$  Siegel upper space  $\pmb{\mathcal{H}}_{g}$ is the set  of
complex $g \times g$ symmetric  matrices with  positive definite imaginary part $\pmb{\mathcal{H}}_{g} = \{\tau \in Mat(g, \mathbb{C}) | \tau = \tau^{t},  \Im(\tau) > 0    \}$.
$\pmb{\mathcal{H}}_{g}$ is isomorphic to the Lie  coset $ \pmb{\mathcal{H}}_g \simeq Sp(2g,\mathbb{R}) / \left( Sp(2g,\mathbb{R}) \cap SO(2g, \mathbb{R}) \right)$.
For a given $m \in Sp(2g,\mathbb{R}) / \left( Sp(2g,\mathbb{R}) \cap SO(2g, \mathbb{R}) \right)$
\beq
m =
 \begin{pmatrix}
  a   &    b \\
  c   &     d  \\
\end{pmatrix},   \nn
\eeq
the bijective map is given by
\beq
\tau(m) = (ai\mathbb{I} + b)(ci\mathbb{I} + d)^{-1}. \label{map}
\eeq

\vspace{.27 cm}

The Iwasawa decomposition
allows to write a symplectic matrix in $Sp(2g,\mathbb{R})$   as  $\pmb{U} \pmb{V} \pmb{K}$,  $\pmb{K} \in  SO(2g, \mathbb{R}) \cap Sp(2g, \mathbb{R})$, $\pmb{V}$  positive definite diagonal  matrix
and $\pmb{U}$  unipotent matrix.
It is convenient  the following $g \times g$ blocks parametrization
\beq
\pmb{V} =  \begin{pmatrix}
 \pmb{V}_{g} & 0 \\
  0   &   \pmb{V}_{g}^{-1}
\end{pmatrix}, \qquad   \pmb{V}_{g} =  diag\left(\sqrt{v_{1}}, \dots , \sqrt{v_{g}}\right), \label{V}
\eeq
for the abelian part with $v_{i} > 0$, $i = 1,\dots, g$,
\beq
\pmb{U} =
\begin{pmatrix}
\pmb{U}_{g} & \, \, \pmb{W}_{g}\pmb{U}_{g}^{-t} \\
0   & \, \, \pmb{U}_{g}^{-t}
\end{pmatrix}, \nn
\eeq
for the unipotent part, with $\pmb{W}_{g}$ symmetric real $g \times g$ matrix

\beq
\pmb{W}_{g} =
\begin{pmatrix}
w_{11}  & w_{12} & \dots  &  w_{1g} \\
 w_{21} & w_{22} & \dots  & w_{2g}  \\
\vdots  & \vdots & \ddots & \vdots  \\
w_{1g}  &  w_{2g}& \dots  &  w_{gg} \\
\end{pmatrix} \label{W}
\eeq
and $\pmb{U}_{g}$ upper unitriangular  real $g \times g$ matrix

\beq
\pmb{U}_{g} =
\begin{pmatrix}
 1    &   u_{12}        & \dots  &   u_{1g}   \\
 0    &   1 & \dots  &      u_{2g}  \\
\vdots  & \vdots & \ddots & \vdots  \\
 0    &  0  & \dots  &  1 \\
\end{pmatrix}. \label{U}
\eeq

With the above parametrization, eq. (\ref{map}) gives the following Iwasawa parametrization for   $\pmb{\mathcal{H}}_{g}$
\beq
\tau_{g}(m) = \pmb{W}_{g} + i\pmb{U}_{g}\pmb{V}_{g}^{2}\pmb{U}_{g}^{t}. \label{Iwcoord}
\eeq

\vspace{.27 cm}

\subsection{ $\pmb{\mathcal{H}}_{g}$  measure in Iwasawa coordinates}

The $\pmb{\Gamma}_{g}$-invariant measure  $d\pmb{\mu}_{g}$ in $\pmb{\mathcal{H}}_{g}$ is given by
\beq
d\pmb{\mu}_{g} = \frac{1}{det(\Im(\tau_{g}))^{g+1}}\prod_{i \le j}^{g} d \, \Re(\tau_{g})_{ij} d \, \Im(\tau_{g})_{ij}, \nn
\eeq
where $\tau_{ij} = \Re(\tau)_{ij} + i \Im(\tau)_{ij}$.

\vspace{.27 cm}

The following two lemmas give $d\pmb{\mu}_{g}$ in Iwasawa coordinates

\vspace{.27 cm}

\begin{lemma}
The following holds true
\beq
det(\Im(\tau_{g}(m))) = \prod_{i=1}^{g} v_{i}, \nn
\eeq
\end{lemma}
\begin{proof}
It follows directly from (\ref{Iwcoord}).
\end{proof}

\vspace{.27 cm}

\begin{lemma}
The Jacobian determinant of  $\tau_{g}(m)$ map given in (\ref{Iwcoord})  is
\beq
J_g = \prod_{i=1}^{g}v_{g - i}^{g -1 - i}. \label{J}
\eeq
\end{lemma}
\begin{proof}
Let us take the following block parametrization
\beq
\pmb{U}_{g} =
\begin{pmatrix}
\pmb{U}_{g -1} &  \pmb{\bar{u}} \\
   \underline{0}  &  1
\end{pmatrix}, \nn
\eeq
where upper line denotes column vectors and lower line row vectors.
$\pmb{U}_{g - 1}$ is a $(g - 1)$-dimensional upper unitriangular real matrix, $\pmb{\bar{u}}$ is a $(g - 1)$-dimensional column vector, $\pmb{\bar{u}} = \pmb{\underline{u}}^t$.
Thus, one has
\eqn
\Im(\tau_{g}) &=&
\pmb{U}_{g}\pmb{V}_{g}^{2}\pmb{U}_{g}^t =
\begin{pmatrix}
\pmb{U}_{g -1} &  \pmb{\bar{u}}  \\
   \underline{0}  &  1
\end{pmatrix}
\begin{pmatrix}
\pmb{V}^{2}_{g -1} &  \bar{0} \\
   \underline{0}  &  v_{g}
\end{pmatrix}
\begin{pmatrix}
\pmb{U}^{t}_{g -1} &  \bar{0} \\
   \pmb{\underline{u}}  &  1
\end{pmatrix}
\nn \\
&=&
\begin{pmatrix}
\pmb{U}_{g -1}\pmb{V}^{2}_{g - 1}\pmb{U}^{t}_{g - 1} + \pmb{\bar{u}}\pmb{\underline{u}}v_g &  \pmb{\bar{u}}v_{g} \\
   \pmb{\underline{u}}v_{g}  &  v_{g}
\end{pmatrix} \nn \\
&=&
\begin{pmatrix}
\Im(\tau_{g - 1}) + \pmb{\bar{u}}\pmb{\underline{u}}v_g &  \pmb{\bar{u}}v_{g} \\
   \pmb{\underline{u}}v_{g}  &  v_{g}
\end{pmatrix}, \nn
\feqn
and therefore
\beq
d \Im(\tau_{g}) = v_{g}^{g - 1} dv_{g} \wedge d \pmb{\underline{u}} \wedge d \Im(\tau_{g - 1}). \nn
\eeq
From which it follows that
\beq
J_{g}  = v_{g}^{g - 1}J_{g -1}, \nn
\eeq
and by iteration one   recovers (\ref{J}).
\end{proof}
\begin{propo}  \label{prop:det}
 The $\pmb{\mathcal{H}}_{g}$ measure  $d\pmb{\mu}_{g}$ in Iwasawa coordinates is given by
  $$d\pmb{\mu}_{g} = \prod_{i= 1}^{g} dv_{i} v_{i}^{i - g - 2} \prod_{i \ge j}^{g} d w_{ij} \prod_{i > j}^{g} d u_{ij}.$$
 \end{propo}

\begin{proof}
\eqn
d \pmb{\mu}_{g} &=& \frac{J_g}{ det(\Im(\tau_{g}))^{g + 1}} \prod_{i=1}^{g}dv_{i} \prod_{i \le j}^{g} d w_{ij} \prod_{i < j}^{g} d u_{ij} \nn \\
&=& \frac{\prod_{i=1}^{g}v_{g - i}^{g -1 - i}}{ \prod_{i=1}^{g}v_{g - i}^{g + 1 }}dv_{g - i} \prod_{i \le j}^{g} d w_{ij} \prod_{i < j}^{g} d u_{ij} \nn \\
 &=& \prod_{i=1}^{g}v_{i}^{i - 2 - g} dv_{i} \prod_{i \le j}^{g} d w_{ij} \prod_{i < j}^{g} d u_{ij}. \nn
\feqn
\end{proof}

\vspace{.2 cm}

\subsection{Eisenstein series}\label{E}

 Let us introduce the following blocks decomposition $\tau_g \in \pmb{\mathcal{H}}_g$
\beq
\tau_g =
\begin{pmatrix}
\tau_{11} & \tau_{12} \\
\tau_{12}^{t}    & \tau_{22}
\end{pmatrix} \nn
\eeq
where  $\tau_{11} \in \pmb{\mathcal{H}}_r$, $\tau_{22} \in \pmb{\mathcal{H}}_{g - r}$.

The $(g - r)$ corank  Eisenstein series, associated to the $\mathbb{F}_{g - r}$   component of the boundary of  $\pmb{\mathcal{H}}_{g}$,
is defined by
\beq
\pmb{E}_{g, r}(\tau,s) =  \sum_{\pmb{\Gamma}_{g} \cap P_{g, r}  \backslash \pmb{\Gamma}_{g}} \left( \frac{det(\Im(\gamma(\tau)))}{det(\Im(\gamma (\tau)_{22}))}  \right)^{s}, \label{Er}
\eeq
where $P_{g, r} \subset Sp(2g,\mathbb{R})$ is the parabolic subgroup which stabilizes the $\mathbb{F}_{g - r}$.

\vspace{.27 cm}

For our purposes it is useful the knowledge of the analytic properties of the $r =  1$ Eisenstein series,
given by the following  proposition [Ya]:

\vspace{.27 cm}

\begin{propo}\label{Yama}
The $r = 1$ Eisenstein series $\pmb{E}_{g,1}(\tau, s)$ of the family given in (\ref{Er})

\beq
\pmb{E}_{g, 1}(\tau,s) =  \sum_{\pmb{\Gamma}_{g} \cap P_{g, 1}  \backslash \pmb{\Gamma}_{g}} \left( \frac{det(\Im(\gamma(\tau)))}{det(\Im(\gamma (\tau)_{22}))}  \right)^{s},   \qquad \Re(s) > g,    \label{Er1}
\eeq
 can be analytically  continued to the  full $s$ plane to  a   meromorphic  function  with
a simple pole in $s = g$ with residue $\frac{1}{2\pmb{\zeta}^{*}(2g)}$, and poles in $s = \frac{\rho}{2}$, where $\rho 's$ are
the non trivial zeros of the Riemann zeta function, $\pmb{\zeta}^{*}(\rho) = \pi^{-\rho/2}\pmb{\Gamma}(\frac{\rho}{2})\pmb{\zeta}(\rho) = 0$.
\end{propo}

\vspace{.35 cm}

\subsection{ $\pmb{\mathcal{H}}_{g - 1} \hookrightarrow  \pmb{\mathcal{H}}_{g}$ embedding}

\begin{lemma}\label{Hembed}( $\pmb{\mathcal{H}}_{g - 1} \hookrightarrow  \pmb{\mathcal{H}}_{g}$ embedding).
Given $\tau_{g} \in  \pmb{\mathcal{H}}_{g}$ and $\tau_{g - 1} \in  \pmb{\mathcal{H}}_{g -1}$, the following decomposition holds
\beq
\Im(\tau_{g}) =
 \begin{pmatrix}
 v_{1}  + \pmb{\underline{u}} \pmb{V}_{g - 1}^{2} \pmb{\bar{u}}   &  \pmb{\underline{u}}\pmb{V}_{g - 1}^{2}\pmb{U}_{g - 1}^{t} \\
 \pmb{U}_{g - 1}\pmb{V}_{g - 1}^{2} \pmb{\bar{u}}   & \Im(\tau_{g -1})
\end{pmatrix}. \nn
\eeq
\end{lemma}

\begin{proof}
\eqn
\Im(\tau_{g}) &=&
\begin{pmatrix}
1  & \pmb{\underline{u}} \\
\bar{0} & \pmb{U}_{g -1}
\end{pmatrix}
\begin{pmatrix}
v_1  & \underline{0} \\
\bar{0} & \pmb{V}_{g -1}^{2}
\end{pmatrix}
\begin{pmatrix}
1  & \underline{0} \\
\pmb{\bar{u}} & \pmb{U}_{g -1}^{t}
\end{pmatrix} \nn \\
&=&
\begin{pmatrix}
 v_{1}  + \pmb{\underline{u}} \pmb{V}_{g - 1}^{2} \pmb{\bar{u}}   &  \pmb{\underline{u}} \pmb{V}_{g - 1}^{2} \pmb{U}_{g - 1}^{t} \\
 \pmb{U}_{g - 1}\pmb{V}_{g - 1}^{2} \pmb{\bar{u}}   &         \pmb{U}_{g - 1} V^{2}_{g - 1} \pmb{U}_{g - 1}^{t}
\end{pmatrix} \nn \\
&=&
\begin{pmatrix}
v_{1}  + \pmb{\underline{u}} \pmb{V}_{g - 1}^{2} \pmb{\bar{u}}   &  \pmb{\underline{u}}\pmb{V}_{g - 1}^{2}\pmb{U}_{g - 1}^{t} \\
\pmb{U}_{g - 1}\pmb{V}_{g - 1}^{2} \pmb{\bar{u}}   & \Im(\tau_{g -1})
\end{pmatrix}. \nn
\feqn
\end{proof}

From the previous result one has also
the following blocks  decomposition for $\tau_{g}$ in terms of $\pmb{\mathcal{H}}_1$ and $\pmb{\mathcal{H}}_{g - 1}$ subspaces:
\begin{propo} \label{propo:corank1}
\beq
\tau_{g} =
\begin{pmatrix}
w_{11} + i(v_{1}  + \pmb{\underline{u}} \pmb{V}_{g - 1}^{2} \pmb{\bar{u}})   & \pmb{\underline{w}} +  i \pmb{\underline{u}}\pmb{V}_{g - 1}^{2}\pmb{U}_{g - 1}^{t} \\
\pmb{\bar{w}}  +   i \pmb{U}_{g - 1}\pmb{V}_{g - 1}^{2} \pmb{\bar{u}}   &   \tau_{g -1}.
\end{pmatrix}. \nn
\eeq
\end{propo}

\vspace{.27 cm}

\subsection{$\pmb{\Gamma}_{g - 1} \hookrightarrow P_{g,1} \subset \pmb{\Gamma}_{g}$ parabolic embedding}\label{Section:parabolic}

A  matrix in $P_{g, 1} \cap \pmb{\Gamma}_{g} \subset Sp(2g,\mathbb{Z})$
has the following  form

\vspace{.27 cm}

\beq
\begin{pmatrix}
 1 &  m &   q  &  n  \\
 0 &  a &  n^{t}  &  b \\
 0 &   0 &  1  &  0\\
 0 &  c &  - m^{t}  &  d
\end{pmatrix},
\qquad
\begin{pmatrix}
a & b \\
c & d
\end{pmatrix}
 \in \pmb{\Gamma}_{g -1},
\,\, m,n \in Mat(1\times (g - 1), \mathbb{Z}),
\,\, q \in \mathbb{Z}. \nn
\eeq

\vspace{.27 cm}

It is useful the following decomposition for the elements in $P_{g, 1} \cap \pmb{\Gamma}_{g}$
(see for example [HKW]):

\begin{propo} \label{Par}
Every matrix in $P_{g, 1} \cap \pmb{\Gamma}_{g}$ can be decomposed as follows:

\vspace{.27 cm}

\beq\label{matrix}
\begin{pmatrix}
 1 &  m &   q  &  n  \\
 0 &  a &  n^{t}  &  b \\
 0 &   0 &  1  &  0\\
 0 &  c &  - m^{t}  &  d
\end{pmatrix}
= g_1 \cdot g_2 \cdot g_3,
\eeq
with
\beq\label{g1}
g_1 =
\begin{pmatrix}
 1 & 0 & 0 & 0 \\
 0 &  a & 0 &  b \\
 0 & 0 & 1 & 0 \\
 0 & c & 0 & d
\end{pmatrix} ,
\qquad
\begin{pmatrix}
a & b \\
c & d
\end{pmatrix}
 \in \pmb{\Gamma}_{g -1},
\eeq
\beq
g_2 =
\begin{pmatrix}
 1 & m & 0 & n \\
 0 & 1 & n^t & 0 \\
 0 & 0 & 1 & 0\\
 0 & 0 & - m^t & 1
\end{pmatrix}
\qquad
m,n \in Mat(1\times (g - 1), \mathbb{Z}),
\eeq
and
\beq
g_3 =
\begin{pmatrix}
 1 & 0 & q & 0 \\
 0 & 1 & 0 & 0 \\
 0 & 0 & 1 & 0 \\
 0 & 0 & 0 & 1
\end{pmatrix}
\qquad
q \in \mathbb{Z}.
\eeq   \nn
\end{propo}

\vspace{.27 cm}

\begin{propo}\label{prop:g1}
The action of $g_1$,$g_2$,$g_3$ on $\tau \in \pmb{\mathcal{H}}_g$

\vspace{.27 cm}

\beq
\tau =
\begin{pmatrix}
\tau_1 & \tau_2  \\
\tau_{2}^{t} & \tau_3
\end{pmatrix}
\qquad  \tau_1 \in \pmb{\mathcal{H}}_{1}, \tau_{3} \in \pmb{\mathcal{H}}_{g - 1 },  \, \tau_2 \in Mat(1\times (g-1), \mathbb{C}), \nn
\eeq
is given by:
\begin{eqnarray}
&& g_1(\tau) =
\begin{pmatrix}
\tau_1 -  \tau_{2}(c\tau_3 + d)^{-1}c\tau_2^t &  *  \\
(c\tau_3 + d)^{-1} \tau_{2}^t  & (a\tau_3 + b)(c\tau_3 + d)^{-1}
\end{pmatrix}, \nn \\
&& g_2(\tau) =
\begin{pmatrix}
 \tau_1^{'} &  *  \\
 \tau_{2}^{t} + m^t \tau_1 + n^t   &   \tau_{3}
\end{pmatrix}, \qquad   \tau_{1}^{'} = \tau_{1}  + m\tau_{3}m^{t} + m^t \tau_{2} + (m^t\tau_{2})^{t} + n m^{t}, \nn \\
&& g_3(\tau) =
\begin{pmatrix}
 \tau_1+q &  *  \\
 \tau_{2}^{t}    &   \tau_{3}
\end{pmatrix} \nn
\end{eqnarray}
where the entries $*$ are given by symmetry of $\tau$.
\end{propo}

\vspace{.35 cm}

\section{\Large{\bf{Proofs of theorem $1$ and theorem $2$}}}\label{SectionErg}

\vspace{.35 cm}

\subsection{Unfolding of the modular integral}
Given a function $f = f(\tau)$ on $\pmb{\mathcal{H}}_g$ invariant under the modular group $\pmb{\Gamma}_g \sim Sp(2g,\mathbb{Z})$,
let us consider the following  Rankin-Selberg type modular integral

 \vspace{.27 cm}

 \beq
 \pmb{I}_{g,1}(s) = \int_{ \pmb{\Gamma}_g \backslash \pmb{\mathcal{H}}_g} d\pmb{\mu}_{g} \, f(\tau) \pmb{E}_{g,1}(\tau_{g}, s), \label{IRS}
 \eeq

 \vspace{.27 cm}

 where    $\pmb{E}_{g, 1}(\tau_{g},s)$ is the non-holomorphic $g - 1$-corank  Eisenstein series,
 introduced in \ref{Er1}, section  \S \ref{E}

 \beq
 \pmb{E}_{g,1}(\tau_{g}, s)  =  \sum_{\pmb{\Gamma}_{g} \cap P_{g, 1}  \backslash \pmb{\Gamma}_{g}} \left( \frac{det(\Im(\gamma(\tau)))}{det(\Im((\gamma (\tau)_{22}))}  \right)^{s},   \qquad \Re(s) > g  \label{Er2}
 \eeq
 related to the    $(g - 1)$ corank component $\mathbb{F}_{g - 1}$ of the boundary of the modular domain $\pmb{\Gamma}_g \backslash \pmb{\mathcal{H}}_g$.

\vspace{.35 cm}

Under suitable growing  conditions for $f$ at the boundary, that are stated as sufficient conditions in
theorem \ref{theorem},
it is possible in (\ref{IRS})  to exchange integration on the modular domain with the sum over the modular transformations $\gamma$'s
appearing  in the Eisenstein series (\ref{Er2}).
This operation allows to unfold the original integration domain  $\pmb{\Gamma}_g \backslash \pmb{\mathcal{H}}_g$ into the
larger  domain  $ (P_{g,1} \cap  \pmb{\Gamma}_{g}) \backslash \pmb{\mathcal{H}}_g$.
As we shall see, this latter integration domain  has simplified features
which becomes  transparent  in Iwasawa coordinates.

\vspace{.35 cm}

Whenever it is allowed to exchange the sum with the integral, for $\pmb{I}_{g,1}(s)$ one finds

 \beq
\pmb{I}_{g,1}(s) = \int_{0}^{\infty} d v_{1} v_{1}^{s  - g - 1}
\int_{( P_{g,1} \cap \pmb{\Gamma}_g ) \backslash \pmb{\mathcal{H}}_g}    d\pmb{\mu}_{g - 1} \int d w_{11} \, d \pmb{\underline{w}} \, d \pmb{\underline{u}} \,  f
\begin{pmatrix}
w_{11} + i(v_{1}  + \pmb{\underline{u}} \pmb{V}_{g - 1}^{2} \pmb{\bar{u}})   & \pmb{\underline{w}} +  i \pmb{\underline{u}}\pmb{V}_{g - 1}^{2}\pmb{U}_{g - 1}^{t} \\
\pmb{\bar{w}}  +   i \pmb{U}_{g - 1}\pmb{V}_{g - 1}^{2} \pmb{\bar{u}}   &   \tau_{g -1}
\end{pmatrix},
\label{Is}
\eeq
 where   integration along  $\pmb{\underline{w}}$ and $\pmb{\underline{u}}$ takes into account identifications by
 the parabolic subgroup $P_{g, 1}$ given in proposition \ref{prop:g1}.

\vspace{.27 cm}

Let us notice that $\pmb{I}_{g,1}(s)$ involves a Mellin integral transform in the abelian Iwasawa coordinate $v_{1} \in \mathbb{R}_{> 0}$.
If certain conditions for the existence of the inverse Mellin transform are fulfilled, then, (by using  proposition \ref{propomellin1}), one  gets
the following asymptotic

\vspace{.27 cm}

\eqn
&&\lim_{v_{1} \rightarrow 0}
\int_{( P_{g,1} \cap \pmb{\Gamma}_g ) \backslash \pmb{\mathcal{H}}_g}    d\pmb{\mu}_{g - 1} \int  d \pmb{\underline{w}} \, d \pmb{\underline{u}} f
\begin{pmatrix}
w_{11} + i(v_{1}  + \pmb{\underline{u}} \pmb{V}_{g - 1}^{2} \pmb{\bar{u}})   & \pmb{\underline{w}} +  i \pmb{\underline{u}}\pmb{V}_{g - 1}^{2}\pmb{U}_{g - 1}^{t}  \\
\pmb{\bar{w}}  +   i \pmb{U}_{g - 1}\pmb{V}_{g - 1}^{2} \pmb{\bar{u}}   &   \tau_{g -1}
\end{pmatrix} \nn \\
&=&
\frac{1}{2\pmb{\zeta}^{*}(2g)} \int_{ \pmb{\Gamma}_g \backslash \pmb{\mathcal{H}}_g} d\pmb{\mu}_{g} \, f(\tau) \nn \\
&=&
\frac{Vol(\pmb{\mathcal{D}}_{g - 1})}{2Vol(\pmb{\mathcal{D}}_{g})}\int_{ \pmb{\Gamma}_g \backslash \pmb{\mathcal{H}}_g} d\pmb{\mu}_{g} \, f(\tau). \nn \\ \label{sugg}
\feqn

\vspace{.27 cm}

Last line of (\ref{sugg}) follows from the formula
$Vol(\pmb{\mathcal{D}}_{g}) = 2\prod_{k = 1}^{g}\pmb{\zeta}^{*}(2k)$ for  the volume of a fundamental region $\pmb{\mathcal{D}}_{g} \simeq \pmb{\Gamma}_{g}
\backslash  \pmb{\mathcal{H}}_{g}$ of the modular group $\pmb{\Gamma}_g$ in $\pmb{\mathcal{H}}_g$.

\vspace{.27 cm}

\subsection{Proof of Theorem $1$.}

The above discussion and eq. (\ref{sugg}) are suggestive of the existence
 of a $\pmb{\mathcal{H}}_{g} \rightarrow  \pmb{\mathcal{H}}_{g - 1}$ reduction for modular integral of automorphic functions through
 the operation of averaging along  unipotent directions $w_{11}, \pmb{\underline{w}}, \pmb{\underline{u}}$ defined in \ref{Iwcorank1}.
  The above argument is turned  into a rigorous proof by the following:

\vspace{.27 cm}

\begin{theorem2} 
Given a $\pmb{\Gamma}_g$-invariant automorphic function $f = f(\tau)$,  let us consider the unipotent average
\beq
\pmb{<} f \pmb{>}_{v_1}(\tau_{g - 1}) : = \int_{\mathbb{R}^{2g - 1}} d w_{11} d \pmb{\underline{w}} \, d \pmb{\underline{u}} \, f(\tau_g ), \nn
\eeq
where $\tau_g$ is given in Iwasawa coordinates according to the corank $(g - 1)$ decomposition given by (\ref{Iwcorank1}) and in  proposition \ref{propo:corank1}.

\vspace{ .27 cm}

The integral  function $\pmb{<} f \pmb{>}_{v_1}(\tau_{g - 1})$ on $\mathbb{R}_{> 0} \times \pmb{\mathcal{H}}_{g - 1}$
is invariant under the genus $(g - 1)$ modular group  $\pmb{\Gamma}_{g - 1}$:

\beq
\pmb{<} f \pmb{>}_{v_1}( (a\tau_{g - 1} + b)(c\tau_{g-1} + d)^{-1}) = \pmb{<} f \pmb{>}_{v_1}(\tau_{g - 1}), \qquad  \begin{pmatrix} a & b \\ c & d \end{pmatrix} \in \pmb{\Gamma}_{g - 1}. \nn
\eeq
\end{theorem2}

\vspace{.35 cm}

\begin{proof}
The  action of $\pmb{\Gamma}_{g-1}$ on $\pmb{<} f \pmb{>}_{v_1}(\tau_{g - 1})$ is provided by the embedding of $\pmb{\Gamma}_{g-1} \hookrightarrow \pmb{\Gamma}_g$ defined by $g_1$
as in Proposition \ref{prop:g1}. As $f$ is $\pmb{\Gamma}_g$ invariant, the proof will follow from the fact that the measure
$ d w_{11} d \pmb{\underline{w}} \, d \pmb{\underline{u}}$ is $\pmb{\Gamma}_{g-1}$-invariant, and that the action of $\pmb{\Gamma}_{g-1}$ over the Siegel
space lives $v_1$ invariant. This permits to reabsorb the transformation in a change of variables which leaves
 the expression of $\pmb{<} f \pmb{>}_{v_1}(\tau_{g - 1})$   invariant
in form. \\

\vspace{.35 cm}

Let us first consider the action of $g_1$, given by (\ref{g1}), on the coset space $\pmb{\mathcal{H}}_{g}$.
Using the Iwasawa construction, the generic point of the coset has the form
\begin{eqnarray}
x:=\begin{pmatrix}
\pmb{U}_{g} \pmb{V}_{g} & \pmb{W}_{g} (\pmb{U}_{g} \pmb{V}_{g})^{-t} \\
\mathbb{O}_g & (\pmb{U}_{g} \pmb{V}_{g})^{-t}
\end{pmatrix} \label{x}
\end{eqnarray}

so that, in particular, its first column is $(\sqrt v_1, 0, \ldots, 0)^t$
\begin{eqnarray}
x =\begin{pmatrix}
 \sqrt{v_{1}} & \dots \\
  \bar{0}        &   \dots     \\
   \end{pmatrix}. \nn
\end{eqnarray}

\vspace{.35 cm}

By acting  on $x$  from the left with $g_1$, one finds the following structure

\beq
g_{1} x =
\begin{pmatrix}
 1 &  \pmb{\underline{0}} & 0 & \pmb{\underline{0}} \\
\pmb{\bar{0}} & a & \pmb{\bar{0}} & b  \\
 0 & \pmb{\underline{0}} & 1 & \pmb{\underline{0}} \\
 \pmb{\bar{0}} & c & \pmb{\bar{0}} & d
\end{pmatrix}
\begin{pmatrix}
 \sqrt{v_{1}} & \dots \\
      \pmb{\bar 0}        &    \dots \\
   0        &    \dots \\
  \pmb{\bar 0}        &    \dots \\
   \end{pmatrix}
=
\begin{pmatrix}
 \sqrt{v_{1}} & \pmb{\underline{r}}^{(1)} \\
  0        &   \pmb{\underline{r}}^{(2)}      \\
  \vdots       &    \vdots \\
  0         &     \pmb{\underline{r}}^{(2g)}         \\
   \end{pmatrix},
\qquad
\begin{pmatrix}
 a & b \\
c &  d  \\
\end{pmatrix}
\in \pmb{\Gamma}_{g - 1}. \label{g1x}
\eeq

\vspace{.35 cm}

In particular,  the $(2g - 1)$ vectors $\pmb{\underline{r}}^{(j)}$, $j = 2,\dots,2g$ are linearly independent, since $\det(g_{1} x) \ne 0$.
The symplectic matrix  $ g_1 x \in Sp(2g,\mathbb{R})$ in (\ref{g1x})
 is no more in the quotient $Sp(2g,\mathbb{R})/(Sp(2g,\mathbb{R})\cap SO(2g,\mathbb{R}))$,
 since it does not have the blocks structure  (\ref{x}).

\vspace{.35 cm}

 However, by  multiplying $g_{1}x$ from the right by an orthosymplectic  matrix $\pmb{K} \in Sp(2g,\mathbb{R})\cap SO(2g,\mathbb{R})$,
 \begin{eqnarray}
\pmb{K} =\begin{pmatrix}
\pmb{A} & \pmb{B} \\ -\pmb{B} & \pmb{A}
\end{pmatrix}, \qquad \pmb{A}^t \pmb{B}=\pmb{B}^t \pmb{A}, \quad \pmb{A}^t \pmb{A} + \pmb{B}^t \pmb{B} =\mathbb{I}, \nn
\end{eqnarray}
 one can determine the $g_{1}x$ coset representative

\beq
g_{1} x K =
\begin{pmatrix}
 \sqrt{v_{1}} & \pmb{\underline{r}}^{(1)} \\
  0        &   \pmb{\underline{r}}^{(2)}      \\
  \vdots       &    \vdots \\
  0         &     \pmb{\underline{r}}^{(2g)}         \\
   \end{pmatrix}
\begin{pmatrix}
 A_{11} & \dots \\
  \pmb{\bar{a}}        &   \dots     \\
   \end{pmatrix}
=
\begin{pmatrix}
 \sqrt{\tilde{v}_{1}} & \dots \\
  \pmb{\bar{0}}        &   \dots     \\
   \end{pmatrix},
\qquad
\pmb{\bar{a}} := (A_{21},\dots A_{g1},-B_{11}\dots -B_{g1})^{t}. \nn
\eeq

\vspace{.35 cm}

In particular, equality for the elements in the  first columns of the previous equation gives

\beq
\begin{pmatrix}
\sqrt{v_{1}}A_{11} +  \pmb{\underline{r}}^{(1)}\pmb{\bar{a}}     \\
\pmb{\underline{r}}^{(2)}\pmb{\bar{a}}      \\
\vdots \\
\pmb{\underline{r}}^{(2g)}\pmb{\bar{a}}         \\
\end{pmatrix}
   =
\begin{pmatrix}
\sqrt{\tilde{v}_{1}}     \\
 0      \\
\vdots \\
0         \\
\end{pmatrix}, \nn
\eeq
and, since  the $(2g - 1)$ vectors $\pmb{\underline{r}}^{(j)}$, $j = 2,\dots,2g$ are linearly independent,
one finds $\pmb{\bar a} = \pmb{\bar 0}$.
This implies  that
\begin{eqnarray}
A_{j1}=\delta_{j1} \sqrt{\tilde v_1/ v_1}, \quad B_{j_1}=0, \qquad j = 1,\dots,g, \nn
\end{eqnarray}
and by using $(A^t A+B^t B)_{11}=1$ one then  gets $\tilde v_1=v_1$.\\

\vspace{.35 cm}

The new defined coordinates are then such that, in the notation of Proposition \ref{prop:g1},
\begin{eqnarray}\label{change}
g_1\left(\begin{pmatrix}
w_{11} + i(v_{1}  + \pmb{\underline{u}} \pmb{V}_{g - 1}^{2} \pmb{\bar{u}})   & \pmb{\underline{w}} +  i \pmb{\underline{u}}\pmb{V}_{g - 1}^{2}\pmb{U}_{g - 1}^{t} \\
\pmb{\bar{w}}  +   i \pmb{U}_{g - 1}\pmb{V}_{g - 1}^{2} \pmb{\bar{u}}   &   \tau_{g -1}
\end{pmatrix} \right)
=\begin{pmatrix}
\tilde w_{11} + i(v_{1}  + \underline{\tilde {u}}  \pmb{\tilde V}_{g - 1}^{2} \pmb{\bar{\tilde u}})   & \pmb{\underline{\tilde {w}}} +  i \pmb{\underline{\tilde u}}
 \pmb{\tilde V}_{g - 1}^{2} \pmb{\tilde U}_{g - 1}^{t} \\
\pmb{\bar{\tilde w}}  +   i \pmb{\tilde U}_{g - 1} \pmb{\tilde V}_{g - 1}^{2} \pmb{\bar{\tilde u}}   &   \tilde \tau_{g -1}
\end{pmatrix}. \nn \\
\end{eqnarray}

\vspace{.35 cm}

Notice that $\pmb{\tilde V}_{g-1}$ and $\pmb{\tilde U}_{g-1}$ are defined by $\tilde \tau_{g-1}$, that does not depends on $\pmb{\bar u}$ and $\pmb{\bar v}$, so that
the transformation of coordinates $(w_{11}, \pmb{\bar u},  \pmb{\bar v})\mapsto (\tilde w_{11}, \pmb{\bar {\tilde u}}, \pmb{\bar {\tilde v}})$ is defined
by the components $g_1(x)_{1j}$, $j=1,\ldots, g$ of relation (\ref{change}). This gives the linear transformation

\vspace{.35 cm}

\begin{eqnarray}
&& \tilde w_{11} +i (\tilde v_1 + \pmb{\underline{\tilde u}} \pmb{\tilde V}^2_{g-1} \pmb{ \bar{\tilde u}} )=
w_{11} +i (v_1 + \pmb{\underline{u}} \pmb{V}^2_{g-1} \pmb{\bar {u}}) -(\pmb{\underline w} + i \pmb{\underline u} \pmb{V}^2_{g-1} \pmb{U}^t_{g-1}) (c\tau_{g-1}+d)^{-1} c
(\pmb{\bar w} +i \pmb{U}_{g-1} \pmb{V}^2_{g-1} \pmb{\bar u}), \cr
&& (\pmb{\underline{\tilde w}} +i \pmb{\underline{ \tilde u}} \pmb{V}^2_{g-1} \pmb{\tilde U}^t_{g-1} )=
(\pmb{\underline{w}} +i \pmb{\underline{u}} \pmb{V}^2_{g-1} \pmb{U}^t_{g-1} )(c\tau_{g-1}+d)^{-t}. \nn \\
\end{eqnarray}

\vspace{.35 cm}

By differentiating and by taking the determinant  one thus gets

\vspace{.35 cm}

\begin{eqnarray}
d\tilde w_{11} d\pmb{\bar {\tilde w}}^{g-1} d \pmb{\bar {\tilde u}}^{g-1} \det ( \pmb{\tilde V}^2_{g-1})=
d w_{11} d \pmb{\bar {w}}^{g-1} d\pmb{\bar {u}}^{g-1} \det (\pmb{V}^2_{g-1}) /|\det (c\tau_{g-1} +d)|^2, \nn
\end{eqnarray}

\vspace{.35 cm}

where we have used that $\det \pmb{\tilde U}_{g-1}=1$. Now,
\begin{eqnarray}
\det (\pmb{\tilde V}^2_{g-1})= \det(\pmb{\tilde U}_{g-1} \pmb{\tilde V}^2_{g-1} \pmb{\tilde U}_{g-1}^t) =\det \Im (\tilde \tau_{g-1}). \nn
\end{eqnarray}

\vspace{.35 cm}

From

\begin{eqnarray}
&& 2i \Im (\tilde \tau_{g-1})=(a\tau_{g-1} +b)(c\tau_{g-1} +d)^{-1}-(\tau_{g-1}^\dagger c^t + d^t)^{-1} (\tau_{g-1}^\dagger a^t + b^t)\cr
&& \phantom{2i \Im (\tilde \tau_{g-1})}
=(\tau_{g-1}^\dagger c^t + d^t)^{-1}[(\tau_{g-1}^\dagger c^t + d^t)(a\tau_{g-1} +b)-(\tau_{g-1}^\dagger a^t + b^t)(c\tau_{g-1} +d)](c\tau_{g-1} +d)^{-1}\cr
&& \phantom{2i \Im (\tilde \tau_{g-1})}=2i (\tau_{g-1}^\dagger c^t + d^t)^{-1} \Im (\tau_{g-1})(c\tau_{g-1} +d)^{-1}, \nn
\end{eqnarray}

\vspace{.35 cm}

where we have used $a^t d-c^t b=\mathbb {I}$, $a^t c=c^t a$ and $b^t d=d^t b$, one gets

\vspace{.35 cm}
\begin{eqnarray}
\det (\pmb{\tilde V}^2_{g-1})=\det (\pmb{V}^2_{g-1}) /|\det (c\tau_{g-1} +d)|^2, \label{detim}
\end{eqnarray}

\vspace{.35 cm}

which shows the invariance of the measure. The fact that $\tilde v_1$ and $\tilde \tau_{g-1}$ do not depend on $w_{11}, \pmb{\bar w}, \pmb{\bar u}$,
implies that the range of coordinates remains unchanged.\\

\vspace{.35 cm}

In conclusion, we have
\begin{eqnarray}
\pmb{<} f \pmb{>}_{v_1}(\tilde{\tau}_{g - 1})  =\int dw_{11} d \pmb{\underline{w}} \, d \pmb{\underline{u}} \, \, f \left(
g_1\left(\begin{pmatrix}
w_{11} + i(v_{1}  + \pmb{\underline{u}} \pmb{V}_{g - 1}^{2} \pmb{\bar{u}})   & \pmb{\underline{w}} +  i \pmb{\underline{u}}\pmb{V}_{g - 1}^{2}\pmb{U}_{g - 1}^{t} \\
\pmb{\bar{w}}  +   i \pmb{U}_{g - 1}\pmb{V}_{g - 1}^{2} \bar{u}   &   \tau_{g -1}
\end{pmatrix}\right)\right), \nn
\end{eqnarray}
since  $f$ is $\pmb{\Gamma}_g$ invariant, this implies  $\pmb{<} f \pmb{>}_{v_1}(\tilde{\tau}_{g - 1}) = \pmb{<} f \pmb{>}_{v_1}(\tau_{g - 1})$.
\end{proof}

\vspace{.35 cm}

\subsection{Proof of Theorem $2$}

We now give the proof of theorem \ref{theorem}.
 We start by  recalling a standard property
concerning Mellin integral transforms, (Proposition \ref{propomellin1}).
In order to prove Theorem \ref{theorem} we shall also need  Proposition \ref{prop:Eeighenvalue},
whose proof is postponed  to \S \ref{SectionLapl}.

\vspace{.35 cm}

\begin{propo}\label{propomellin1}
Let $\varphi = \varphi(s)$  be the following  meromorphic  function on the $s$ plane

\beq
\varphi(s) = \sum_{i = 1}^{l} \frac{C_i}{(s - s_i)^{n_i}}, \qquad  n_{i} \in \mathbb{N}_{\ge 0}, \nn
\eeq

 then the following identity holds

 \beq
 \frac{1}{2\pi i} \int_{\sigma - i\infty}^{\sigma + i\infty} ds \, y^{-s} \varphi(s) = \sum_{i=1}^{l} (-)^{n_{i}} \frac{C_i}{n_{i}!}\, y^{-s_{i}} \log^{n_i} y, \qquad        \sigma > Max_{i}\{ \Re(s_i) \}. \nn
 \eeq

\end{propo}

\begin{proof}
It is easily obtained by using  residues theorem, and by closing  the integration  contour  such that it contains the points
$s =  s_{i}, i = 1,\dots,l.$
\end{proof}

\vspace{.35 cm}

\vspace{.35 cm}

\begin{theorem2} 
Let $f =  f(\tau)$ a $\pmb{\Gamma}_{g}$-invariant function of rapid decay for $\tau$ going to all the  components  $\mathbb{F}_{g - r}, r = 1,\dots,g-1$  of the
$\pmb{\mathcal{H}}_g$ boundary. Let $f(\tau)$ be differentiable up to second order, with  Laplacian $ \Delta f$  of rapid decay,
then the following asymptotic  holds true:

\vspace{.35 cm}

\beq
 \int_{ \pmb{\mathcal{D}}_{g - 1}}    d\pmb{\mu}_{g - 1} \pmb{<} f \pmb{>}_{v_1} (\tau_{g-1})
 \sim \frac{Vol(\pmb{\mathcal{D}}_{g-1})}{2Vol(\pmb{\mathcal{D}}_{g})}\int_{ \pmb{\mathcal{D}}_g} d\pmb{\mu}_{g} \, f(\tau) +  O(v_{1}^{g - \frac{\Theta}{2}}),
 \quad  v_{1} \rightarrow 0,  \label{equidg}
 \eeq

 \vspace{.35 cm}
where  $\pmb{\mathcal{D}}_g \sim  \pmb{\Gamma}_g  \backslash \pmb{\mathcal{H}}_g $ is a $\pmb{\Gamma}_g$ fundamental domain,
with volume $Vol(\pmb{\mathcal{D}}_g) = 2 \prod_{k=1}^{g}\pmb{\zeta}^{*}(2k)$.
  Integration  along  unipotent coordinates $w_{11}, \pmb{\underline{w}}, \pmb{\underline{u}}$
takes into account the  identifications by the parabolic subgroup $P_{g,1}$, given in proposition \ref{prop:g1},
 and $\Theta : = \sup\{\Re(\rho)| \pmb{\zeta}^{*}(\rho) =  0 \}$, is the superior of the real part of the
 non trivial zeros $\rho$'s of the Riemann zeta function.
\end{theorem2}

\begin{proof}
Let us consider the modular integral

\beq
\pmb{I}_{g,1}(s) = \int_{ \pmb{\Gamma}_g \backslash \pmb{\mathcal{H}}_g} d\pmb{\mu}_{g} \, f(\tau) \pmb{E}_{g,1}(\tau). \nn
\eeq

The Eisenstein series  $\pmb{E}_{g,1}(\tau,s)$ defined in (\ref{Er1}) is of polynomial growth for $\tau$ going to each  component of the $\pmb{\mathcal{H}}_g$ boundary.
  For $\Re(s) > g$ one can use the series representation for the Eisenstein series, and, by Lebesgue  dominated convergence
   one can exchange the series with the modular integral
\eqn
\pmb{I}_{g,1}(s) &=& \int_{ \pmb{\Gamma}_g \backslash \pmb{\mathcal{H}}_g} d\pmb{\mu}_{g} \, f(\tau) \sum_{\pmb{\Gamma}_{g} \cap P_{g, 1}  \backslash \pmb{\Gamma}_{g}} \left( \frac{det(\Im(\gamma(\tau)))}{det(\Im((\gamma (\tau)_{22}))}  \right)^{s}      \nn \\
&=&  \sum_{\pmb{\Gamma}_{g} \cap P_{g, 1}  \backslash \pmb{\Gamma}_{g}}\int_{ \pmb{\Gamma}_g \backslash \pmb{\mathcal{H}}_g} d\pmb{\mu}_{g} \, f(\tau)
 \left( \frac{det(\Im(\gamma(\tau)))}{det(\Im( \gamma( \tau_{22}))}  \right)^{s}  \nn \\
&=&
\int_{   \pmb{\Gamma}_g \cap P_{g, 1}    \backslash \pmb{\mathcal{H}}_g} d\pmb{\mu}_{g} \, f(\tau) \left( \frac{det(\Im(\tau))}{det(\Im( \tau_{22}))}  \right)^{s}. \nn
\feqn

\vspace{.35 cm}

In the last line modular transformations $\gamma$'s in the coset $(\pmb{\Gamma}_{g} \cap P_{g, 1})  \backslash \pmb{\Gamma}_{g}$
are used to unfold the integration domain.

\vspace{.35 cm}

Since $f(\tau)$ is of rapid decay for $\tau$ going at the $\pmb{\mathcal{H}_{g}}$ boundary, the modular integral is uniformly convergent
with respect to the variable $s$, and thus $\pmb{I}_{g,1}(s)$ inherits analytic properties of  the Eisenstein series  $\pmb{E}_{g,1}(s,\tau)$.
Then, due to proposition \ref{Yama},  $\pmb{I}_{g,1}(s)$ can be analytically continued to the full $s$ plane to a meromorphic function
  with a simple pole in $s = g$, and poles in $s = \rho/2$, where $\rho$'s are the non trivial zeros of the Riemann zeta function,
  $\pmb{\zeta}^{*}(\rho) = 0$.
Thus one can write the following expansion

\beq
\pmb{I}_{g,1}(s) = \frac{C_g}{ s - g} + \sum_{\pmb{\zeta}^{*}(\rho) = 0} \frac{C_{\rho}}{ s - \rho /2}, \label{expans}
\eeq
for multiple Riemann  zeros $\rho$'s,  one has to rise the denominator in the above formula by the appropriate power.

\vspace{.27 cm}

In (\ref{expans}), $C_g$ is given by

\beq
C_g = \frac{1}{2\pmb{\zeta}^{*}(2g)} \int_{\pmb{\mathcal{D}}_g} d\pmb{\mu}_g \, f(\tau), \nn
\eeq
since $\pmb{E}_{g,1}(\tau,s)$ has a simple pole in $s = g$ with  residue $1/2\pmb{\zeta}^{*}(2g)$.

\vspace{.35 cm}

Then, by using Iwasawa coordinates one can write $\pmb{I}_{g,1}(s)$ in the following convenient form

\beq
\pmb{I}_{g,1}(s) = \int_{0}^{\infty} dv_1 \, v_{1}^{s - g - 1} \, \int_{  \pmb{\Gamma}_{g - 1}  \backslash \pmb{\mathcal{H}}_g }    d\pmb{\mu}_{g - 1} \int dw_{11} d \pmb{\underline{w}} \, d \pmb{\underline{u}} f
\begin{pmatrix}
w_{11} + i(v_{1}  + \pmb{\underline{u}} \pmb{V}_{g - 1}^{2} \pmb{\bar{u}})   & \pmb{\underline{w}} +  i \pmb{\underline{u}}\pmb{V}_{g - 1}^{2}\pmb{U}_{g - 1}^{t} \\
\pmb{\bar{w}}  +   i \pmb{U}_{g - 1}\pmb{V}_{g - 1}^{2} \pmb{\bar{u}}   &   \tau_{g -1}
\end{pmatrix}, \label{IMellin}
\eeq
where the decomposition in terms of a modular integral over $\pmb{\Gamma}_{g - 1}  \backslash \pmb{\mathcal{H}}_g$
follows from theorem \ref{invariance}.

\vspace{.35 cm}

Equation (\ref{IMellin}) states that  $\pmb{I}_{g,1}(s)$ is the Mellin transform of the following integral function

\beq
v_{1}^{-g} \cdot F_{g,1}(v_1 ) : =   \int_{  \pmb{\Gamma}_{g - 1}  \backslash \pmb{\mathcal{H}}_g }    d\pmb{\mu}_{g - 1} \int dw_{11} d \pmb{\underline{w}} \, d \pmb{\underline{u}} f
\begin{pmatrix}
w_{11} + i(v_{1}  + \pmb{\underline{u}} \pmb{V}_{g - 1}^{2} \pmb{\bar{u}})   & \pmb{\underline{w}} +  i \pmb{\underline{u}}\pmb{V}_{g - 1}^{2}\pmb{U}_{g - 1}^{t} \\
\pmb{\bar{w}}  +   i \pmb{U}_{g - 1}\pmb{V}_{g - 1}^{2} \pmb{\bar{u}}   &   \tau_{g -1}.
\end{pmatrix} \label{F}
\eeq

\vspace{.27 cm}

If the following integral defining the  $\pmb{I}_{g,1}(s)$ inverse Mellin transform
\beq
\pmb{\mathcal{M}}^{-1}[\pmb{I}_{g,1}(s)](y) = \frac{1}{2\pi i}\int_{\sigma - i\infty}^{\sigma + i\infty}ds \, y^{-s}\pmb{I}_{g,1}(s) = \frac{y^{- \sigma}}{2 \pi i} \int_{-\infty}^{\infty}dt \, y^{-it}\pmb{I}_{g,1}(\sigma + it), \label{M}
\eeq
is convergent, then through proposition \ref{propomellin1}, one  obtains the    $v_1 \rightarrow 0$ asymptotic for the function $F_{g,1}(v_1)$.

 \vspace{.27 cm}

 Since $f(\tau)$ is twice differentiable,  we use $\Delta \pmb{E}_{g,1}(\tau, s) = 2^{\frac{g - 1}{g + 1}} s(g - s)\pmb{E}_{g,1}(\tau, s)$,
 (a proof of this result  is  given in  proposition \ref{prop:Eeighenvalue}).

 \vspace{.27 cm}

 By integration by parts one then  finds
\beq
\pmb{I}_{g,1}(s) =  \frac{2^{-\frac{g-1}{g+1}}}{s(g - s)}\int_{\pmb{\mathcal{D}}_g} d \pmb{\mu}_g  \pmb{E}_{g,1}(\tau,s) \pmb{\Delta} f(\tau). \label{I2}
\eeq

This shows that $\pmb{I}_{g,1}( \sigma + it )$ falls off as $O(t^{-2})$ for $t \rightarrow  \pm \infty$,
for all $\sigma$'s where the following integral
\beq
\int_{\pmb{\mathcal{D}}_g} d \pmb{\mu}_g  \pmb{E}_{g,1}(\tau,s) \pmb{\Delta} f(\tau), \nn
\eeq
is  convergent.
Under our assumption that $\pmb{\Delta} f(\tau)$ is of rapid decay for $\tau$ going to the boundary,
the above integral is convergent, since  $\pmb{E}_{g,1}(\tau,s)$ is of polynomial growth
for $\tau$ going to each component of the $\pmb{\mathcal{H}}_g$ boundary.
It follows  that  the integral (\ref{M}) is convergent, and thus $\pmb{\mathcal{M}}^{-1}[\pmb{I}_{g,1}(s)](v_1)$
exists, and $\pmb{\mathcal{M}}^{-1}[\pmb{I}_{g,1}(s)](v_1) = F(v_1)$.

\vspace{.35 cm}

 By  analytic properties of $\pmb{I}_{g,1}(s)$ given  in eq. (\ref{expans}),
except for a simple pole in $s = g$,  $\pmb{I}_{g,1}(s)$ is analytic on $\Re(s) > \frac{\Theta}{2}$,
where $\Theta = Sup\{ \Re(\rho) |  \pmb{\zeta}^{*}(\rho) = 0 \}$.
Then, asymptotic  eq. (\ref{equidg}) including dependence of the  error estimate on $\Theta$,  follows from (\ref{IMellin}) and by using proposition \ref{propomellin1}.
 Finally, the ratio  appearing in eq. (\ref{equidg}) between  volumes of  modular domains
 follows from the formula $Vol(\pmb{\mathcal{D}}_g) = 2\prod_{k=1}^{g} \pmb{\zeta}^{*}(2k) $.

\end{proof}

\vspace{.27 cm}

\subsection{$\pmb{\mathcal{H}}_g$ Laplacian and the  rank $(g - 1)$ Eisenstein series} \label{SectionLapl}
Let $G_{IJ}$ be the $\pmb{\mathcal{H}}_g$ $Sp(2g,\mathbb{R})$-invariant metric, with infinitesimal
 line element  $ds^2 = G_{IJ}dX_{I}dX_{J}$, where $X_{I}$ is  a system of $g(g + 1)$ real  coordinates for $\pmb{\mathcal{H}}_g$.
 We indicate with $G^{IJ}$ the inverse metric of $G_{IJ}$, $G^{IK} \, G_{KJ} = \delta^{I}_{J}$. We also  use the notation  $G : = \det G_{IJ}$,
 for the determinant of the metric.
Let us consider the $\pmb{\mathcal{H}}_g$ Laplacian operator  $\Delta := - \frac{1}{\sqrt{|G|}} \partial_{I} \sqrt{|G|} G^{IJ} \partial_{J}$.
In this section we prove  that the $(g - 1)$ corank Eisenstein series $\pmb{E}_{g,1}(\tau, s)$ defined in \S \ref{E}
is an  eighenfunction of the $\pmb{\mathcal{H}}_g$ Laplacian $\Delta$,

\vspace{.27 cm}

\beq
\Delta \pmb{E}_{g,1}(\tau, s) = 2^{\frac{g - 1}{g + 1}} s(g - s) \pmb{E}_{g,1}(\tau,s). \label{guess2}
\eeq

\vspace{.27 cm}

In order to prove this result, we first need the following   lemma:

\vspace{.27 cm}

\begin{lemma}\label{lemmaMetric}
In Iwasawa coordinates:   $G^{v_i , v_i} =2^{\frac {g-1}{g+1}} v_{i}^{2}$.
\end{lemma}
\begin{proof}
Let us consider the Iwasawa decomposition $\pmb{U}\pmb{V}\pmb{K}$ of $Sp(2g,\RR)$. The elements of the quotient $\pmb{\mathcal{H}}_g$ are represented by the points
$h=\pmb{U}\pmb{V}$. In \cite{gruppi} it has been shown that the invariant metric can then be obtained as

\vspace{.27 cm}

\begin{eqnarray}
ds^2= \kappa \, {\rm Tr} (\pmb{J}\otimes \pmb{J}),\qquad \pmb{J}=\pmb{\Pi}((\pmb{U}\pmb{V})^{-1} d(\pmb{U}\pmb{V}))\label{metric}
\end{eqnarray}

\vspace{.27 cm}

where $\pmb{\Pi}$ projects orthogonally to the space tangent to $\pmb{K}$, and $\kappa$ is a normalization constant.
As $\pmb{K}$ is the intersection with the orthogonal group $SO(2g\RR)$, $\pmb{\Pi}$ takes the symmetric part, so that

\vspace{.27 cm}

\begin{eqnarray}
\pmb{J}=\pmb{V}^{-1} d\pmb{V} +\frac 12 (\pmb{V}^{-1} \pmb{U}^{-1} d\pmb{U} \pmb{V}+ \pmb{V}d\pmb{U}^t \pmb{U}^{-t} \pmb{V}^{-1}). \nn
\end{eqnarray}

\vspace{.27 cm}

In order to compute the trace in (\ref{metric}) note that the parenthesis is the sum of nilpotent matrices (as $\pmb{U}$ unipotent
implies $\pmb{U}^{-1}d\pmb{U}$ nilpotent), whereas $\pmb{V}^{-1} d\pmb{V}$ is diagonal, so that mixed products have vanishing trace and we remain with the terms

\vspace{.27 cm}

\begin{eqnarray}\label{metric-2}
ds^2= \frac {\kappa}2 \sum_{i=1}^g \frac 1{v_i^2} dv_i^2 +\frac {\kappa}2 {\rm Tr}
(\pmb{V}^{-2} \pmb{U}^{-1} d\pmb{U} \pmb{V}^2 d\pmb{U}^t \pmb{U}^{-t}).
\end{eqnarray}

\vspace{.27 cm}

Notice  that there are no off-diagonal terms of the form $d\pmb{V} \otimes d\pmb{U}$, so that from (\ref{metric-2}) we get

\vspace{.27 cm}

\begin{eqnarray}
G^{v_i v_i}=\frac 2\kappa v_i^2, \qquad i=1,\ldots, g.  \nn
\end{eqnarray}

\vspace{.27 cm}

To compute $\kappa$ we can compute the determinant of the metric (\ref{metric-2}) and compare the result with Proposition \ref{prop:det}.
We know that the determinant does not depend on the coordinates $u_{ij}$ and $w_{ij}$ in $\pmb{U}$ so that we can compute it for
$u_{ij}=0$ and $w_{ij}=0$. This gives

\vspace{.27 cm}

\begin{eqnarray}
\pmb{V}^{-2} d\pmb{U} =\begin{pmatrix}
\pmb{V}_g^{-2} d\pmb{U}_g & \pmb{V}_g^{-2} d\pmb{W}_g \\ 0 & -\pmb{V}_g^2 d\pmb{U}_g^t
\end{pmatrix},
\qquad
\pmb{V}^{2} d\pmb{U}^t =\begin{pmatrix}
\pmb{V}_g^{2} d\pmb{U}_g^t & 0 \\ \pmb{V}_g^{-2} d\pmb{W}_g & -\pmb{V}_g^{-2} d\pmb{U}_g
\end{pmatrix}.  \nn
\end{eqnarray}

\vspace{.27 cm}

Then

\vspace{.27 cm}

\begin{eqnarray}
{\rm Tr}(\pmb{V}^{-2} \pmb{U}^{-1} d\pmb{U} \pmb{V}^2 d\pmb{U}^t \pmb{U}^{-t})=
{\rm Tr}[2\pmb{V}_g^2 d\pmb{U}^t_g \pmb{V}_g^{-2} d\pmb{U}_g+ \pmb{V}_g^{-2} d\pmb{W}_g \pmb{V}_g^{-2} d\pmb{W}_g].  \nn
\end{eqnarray}

\vspace{.27 cm}

First, notice  that

\begin{eqnarray}
(\pmb{V}_g^{-2} d\pmb{W}_g \pmb{V}_g^{-2})_{ij}= \frac 1{v_i v_j} dw_{ij}. \nn
\end{eqnarray}

Thus

\begin{eqnarray}
\pmb{A} :={\rm Tr}(\pmb{V}_g^{-2} d\pmb{W}_g \pmb{V}_g^{-2} d\pmb{W}_g)= \sum_{i,j} \frac 1{v_i v_j} dw_{ij} dw_{ji}=\sum_{i=1}^g \frac 1{v_i^2} dw_{ii}^2
+2\sum_{i<j} \frac 1{v_i v_j} dw_{ij}^2.  \nn
\end{eqnarray}

Then, this part of the metric is diagonal and contributes to the determinant with the term

\vspace{.27 cm}

\begin{eqnarray}
{\det} \pmb{A} =2^{g(g-1)/2} \prod_{1\leq i\leq j \leq g} \frac 1{v_i v_j} =2^{g(g-1)/2} \prod_{i=1}^g \frac 1{v_i^g}.   \nn
\end{eqnarray}

\vspace{.27 cm}
In the same way we get

\vspace{.27 cm}

\begin{eqnarray}
\pmb{B} :={\rm Tr}[2\pmb{V}_g^2 d\pmb{U}^t_g \pmb{V}_g^{-2} d\pmb{U}_g]=2\sum_{1\leq i<j\leq g} \frac {v_j}{v_i} d\pmb{U}_{ij}^2.  \nn
\end{eqnarray}

\vspace{.27 cm}

Again, this is diagonal and it contributes to the determinant with the term
\begin{eqnarray}
&& \det \pmb{B} =2^{g(g-1)/2} \prod_{1\leq i<j\leq g} \frac {v_j}{v_i} =2^{g(g-1)/2} \prod_{1\leq i<j\leq g} \frac 1{v_i v_j}
\prod_{1\leq i<j\leq g} \frac 1{v_i^2}\cr \nn
&&\phantom{\det B}=2^{g(g-1)/2} \prod_{i=1}^g \frac 1{v_i^g} \prod_{i=1}^g v_i^{2(i-1)}=2^{g(g-1)/2} \prod_{i=1}^g v_i^{2i-2-2g}. \nn
\end{eqnarray}

\vspace{.27 cm}

The term
\begin{eqnarray}
\pmb{C} :=\sum_{i=1}^g \frac 1{v_i^2} dv_i^2 \nn
\end{eqnarray}

\vspace{.27 cm}

gives the contribution
\begin{eqnarray}
\det \pmb{C} =\prod_{i=1}^g \frac 1{v_i^2}, \nn
\end{eqnarray}

\vspace{.27 cm}

and by taking into account the factor $\kappa/2$ we finally have
\begin{eqnarray}
G=\left(\frac \kappa2 \right)^{g(g+1)} \det \pmb{A} \det \pmb{B} \det  \pmb{C} =\left(\frac \kappa2 \right)^{g(g+1)} 2^{g(g-1)}
\prod_{i=1}^g v_i^{2i-4-2g}. \nn
\end{eqnarray}

\vspace{.27 cm}

Comparing with Proposition \ref{prop:det} gives
\begin{eqnarray}
\frac \kappa2 =2^{-\frac {g-1}{g+1}},  \nn
\end{eqnarray}
therefore one finally gets

\vspace{.27 cm}

\begin{eqnarray}
G^{v_i v_i}=2^{\frac {g-1}{g+1}} v_i^2.  \nn
\end{eqnarray}

\end{proof}

\vspace{.35 cm}

By using  lemma \ref{lemmaMetric}, we are then able to  prove the following proposition

\vspace{.35 cm}

\begin{propo}  \label{prop:Eeighenvalue}
Let $\Delta$  be the $\pmb{\mathcal{H}}_g$ Laplacian operator

\beq
\Delta : =  - \frac{1}{\sqrt{|G|}} \partial_{I} \sqrt{|G|} \, G^{IJ} \partial_{J}, \nn
\eeq

then
\beq
\Delta \pmb{E}_{g,1}(\tau, s) =  2^{\frac {g-1}{g+1}} s(g - s) \pmb{E}_{g,1}(\tau,s). \label{guess3}
\eeq
\end{propo}

\vspace{.27 cm}

\begin{proof}

In Iwasawa coordinates  $\sqrt G =  \prod_{i=1}^{g} v_{i}^{i - g - 2}$.
With the help of lemma \ref{lemmaMetric}, by direct computation in Iwasawa coordinates one finds

\beq
\Delta \left(\frac{\det(\Im(\tau))}{\det(\Im(\tau)_{22})}\right)^s = \Delta v_{1}^s = 2^{\frac {g-1}{g+1}} s(g - s)v_{1}^{s}, \nn
\eeq
then, by $\pmb{\Gamma}_g$-invariance of the Laplacian operator one also has

\beq
\Delta \left(\frac{\det(\Im(\gamma(\tau)))}{\det(\Im( \gamma(\tau))_{22})}\right)^s =  2^{\frac {g-1}{g+1}} s(g - s) \left(\frac{\det(\Im(\gamma(\tau)))}{\det(\Im( \gamma(\tau))_{22})}\right)^s,
\qquad  \gamma \in Sp(2g,\mathbb{Z}), \nn
\eeq

eq. (\ref{guess3}) then follows.

\end{proof}

\end{document}